\documentclass{amsart}
\usepackage[utf8]{inputenc}
\usepackage{amsmath, amssymb, amsfonts, amsthm}
\usepackage{mathtools}
\usepackage{thmtools}
\usepackage{lmodern}
\usepackage{geometry}
\usepackage{float}
\usepackage{array}
\usepackage{algorithm}
\usepackage{graphicx}
\usepackage{fancyhdr}
\usepackage{hyperref}
\usepackage[capitalize]{cleveref}
\usepackage{algpseudocode}
\usepackage{textgreek}
\usepackage{tikz}
\usetikzlibrary{positioning}
\usetikzlibrary{patterns}
\usetikzlibrary{decorations.pathreplacing,calligraphy}
\usepackage{chngcntr}
\usepackage{pgfplots}
\usepackage{fullpage}

\hypersetup{
    colorlinks = true,
    linkcolor = red,
    citecolor = blue,
    linktoc=page,
    urlcolor=red
}
\usepackage{xcolor}

\theoremstyle{plain}
\newtheorem{theorem}{Theorem}[section]
\newtheorem{lemma}[theorem]{Lemma}
\newtheorem{proposition}[theorem]{Proposition}
\newtheorem{conjecture}[theorem]{Conjecture}

\theoremstyle{definition}
\newtheorem{definition}[theorem]{Definition}
\newtheorem{example}[theorem]{Example}

\theoremstyle{remark}

\newcommand{\card}[1]{\lvert #1 \rvert}

\newcommand{\cS}{\mathcal{S}}

\newcommand{\cM}{\mathcal{M}}

\newcommand{\cD}{\mathcal{D}}
\newcommand{\cF}{\mathcal{F}}

\newcommand{\RR}{\mathbb{R}}
\newcommand{\NN}{\mathbb{N}}
\newcommand{\eps}{\varepsilon}

\newcommand{\bin}{\text{Bin}}

\newcommand{\dg}{\delta_{\textnormal{greedy}}}
\newcommand{\tg}{T_{\textnormal{greedy}}}
\newcommand{\av}{\mathcal{AV}}

\renewcommand{\ss}{\mathcal{SS}}

\newcommand{\rss}{\mathcal{RSS}}

\title{Shuffle squares and reverse shuffle squares}
\author{Xiaoyu He}
\address{Department of Mathematics, Princeton University, Princeton, NJ 08540, USA}
\email{xiaoyuh@princeton.edu}

\author{Emily Huang}
\address{Department of Mathematics, Stanford University, Stanford, CA 94305, USA}
\email{ehuang2@stanford.edu}

\author{Ihyun Nam}
\address{Department of Mathematics, Stanford University, Stanford, CA 94305, USA}
\email{ihyun@stanford.edu}

\author{Rishubh Thaper}
\address{Department of Mathematics, Stanford University, Stanford, CA 94305, USA}
\email{rthaper@stanford.edu}

\date{September 2021}

\begin{document}

\counterwithin{equation}{section}

\maketitle

\begin{abstract}
    Let $\ss_k(n)$ be the family of {\it shuffle squares} in $[k]^{2n}$, words that can be partitioned into two disjoint identical subsequences. Let $\rss_k(n)$ be the family of {\it reverse shuffle squares} in $[k]^{2n}$, words that can be partitioned into two disjoint subsequences which are reverses of each other. Henshall, Rampersad, and Shallit conjectured asymptotic formulas for the sizes of $\ss_k(n)$ and $\rss_k(n)$ based on numerical evidence. We prove that
    \[
    \card{\ss_k(n)}=\dfrac{1}{n+1}\dbinom{2n}{n}k^n-\dbinom{2n-1}{n+1}k^{n-1}+O_n(k^{n-2}),
    \]
    confirming their conjecture for $\card{\ss_k(n)}$. We also prove a similar asymptotic formula for reverse shuffle squares that disproves their conjecture for $\card{\rss_k(n)}$. As these asymptotic formulas are vacuously true when the alphabet size is small, we study the binary case separately and prove that $\card{\ss_2(n)} \ge \binom{2n}{n}$.
    
\end{abstract}

\pagenumbering{arabic}

\section{Introduction}

A {\it word of length $n$} over the alphabet $[k]\coloneqq \{0,1,\dots,k-1\}$ is an element of $[k]^n$. A {\it subsequence} of $w$ is any word obtained by deleting zero or more symbols from $w$. A word is a \textit{shuffle square} if it can be partitioned into two disjoint identical subsequences, and a \textit{reverse shuffle square} if it can be partitioned into two disjoint subsequences that are reverses of each other. 

Shuffle squares and reverse shuffle squares were first defined by Henshall, Rampersad, and Shallit \cite{henshall2012} in 2012 in the context of formal language theory, and have arisen naturally in complexity theory and coding theory since then. Independently, Rizzi and Vialette \cite{rizzi2013} and Buss and Soltys \cite{buss2014} showed that for some fixed $k$, deciding whether a word in $[k]^{2n}$ is a shuffle square is NP-complete. Recently, Bulteau and Viallette \cite{bulteau2020} improved this result by showing that even over a binary alphabet, deciding whether a word is a shuffle square is NP-complete. 

The problem of finding the largest shuffle square contained in any binary word of length $n$, called the ``twins in words'' problem, was studied by Axenovich, Person, and Puzynina \cite{apptwins2012}. They developed a regularity lemma for words analogous to Szemer\'edi's graph regularity lemma and used it to show that every binary word of length $n$ contains a shuffle square of length $n-o(n)$. The twins problem can be viewed as asking for the longest common subsequence between a word and itself, and so is closely related to the problem of finding longest common subsequences between distinct words. This problem has been studied extensively in both combinatorics and theoretical computer science \cite{bukh2018, bukh2014, cormen2001, hirschberg1975} and has applications in computational biology through insertions and deletions of nitrogenous base pairs \cite{xia2007}. It also arises naturally in the study of error-correcting codes for deletion channels, which were introduced by Levenshtein \cite{levenshtein1966} in 1966. Recently, the Axenovich-Person-Puzynina regularity lemma introduced in the study of the ``twins in words'' problem, was adapted by Guruswami, He, and Li \cite{guruswami2021} to prove the first nontrivial upper bound on the zero-rate threshold of the binary deletion channel.

Despite all of this work, the basic question of enumerating shuffle squares remained unsolved. Henshall, Rampersad, and Shallit \cite{henshall2012} conjectured asymptotic formulas involving Catalan numbers for $\card{\ss_k(n)}$ and $\card{\rss_k(n)}$ based on numerical evidence, and the enumeration of shuffle squares was called one of the most interesting problems in the area by Bulteau and Viallette \cite{bulteau2020}.

In this paper, we prove asymptotic formulas for $\card{\ss_k(n)}$ and $\card{\rss_k(n)}$. Our first result confirms the conjecture of Henshall, Rampersad, and Shallit for shuffle squares.

\begin{theorem}
\label{theorem:ss}
For $k\ge 2$ and $n\ge 1$, \[\card{\ss_k(n)}=\dfrac{1}{n+1}\dbinom{2n}{n}k^n-\dbinom{2n-1}{n+1}k^{n-1}+O_n(k^{n-2}).\]
\end{theorem}

The coefficient of the $k^n$ term on the right-hand side can be recognized as the $n$th Catalan number $C_n$, and the second-order term turns out to enumerate the total number of valleys across all Dyck paths of semi-length $n$. Indeed, the proof of Theorem \ref{theorem:ss} employs inclusion-exclusion and several Catalan bijections, mapping shuffle squares to standard $2 \times n$ Young tableaux, non-nesting perfect matchings, and Dyck paths.

Next, we prove an analogous formula for reverse shuffle squares, which disagrees with the conjectured formula of Henshall, Rampersad, and Shallit in the second-order term.

\begin{theorem}
\label{theorem:rss}
For $k\ge 2$, we have $\card{\rss_k(1)} = k$ and \[\card{\rss_k(n)}=\dfrac{1}{n+1}\dbinom{2n}{n}k^n-\dfrac{2n^3+9n^2-35n+30}{n^3+3n^2+2n}\dbinom{2n-2}{n-1}k^{n-1}+O_n(k^{n-2})\]
for $n\ge 2$.
\end{theorem}

The coefficient of $k^n$ is again the $n$-th Catalan number. The proof is similar to that of \cref{theorem:ss}, except that here the correct Catalan interpretation maps reverse shuffle squares to $123$-avoiding permutations of length $n$. The coefficient of $k^{n-1}$ above counts the number of unordered pairs of $123$-avoiding permutations of length $n$ differing by a transposition. We remark that in both \cref{theorem:ss} and \cref{theorem:rss}, we expect our methods can be extended to determine any finite number of terms in the asymptotic formulas.

The main drawback of Theorems \ref{theorem:ss} and \ref{theorem:rss} is that they say nothing about $|\ss_k(n)|$ and $|\rss_k(n)|$ when the alphabet size $k$ is small. We conclude with two results for shuffle squares over a binary alphabet. Henshall, Rampersad, and Shallit showed that binary reverse shuffle squares are exactly {\it abelian squares}, those binary words where the first half is a permutation of the second half, thus proving that $|\rss_2(n)|=\binom{2n}{n}$. We show there are more binary shuffle squares than binary reverse shuffle squares for $n\ge 3$.

\begin{theorem}
\label{theorem:bss}
For $n\ge 1$, $\card{\ss_2(n)} \ge \dbinom{2n}{n}$, with strict inequality if $n\ge 3$.
\end{theorem}

The proof of Theorem \ref{theorem:bss} uses a greedy algorithm to identify $\binom{2n}{n}$ binary shuffle squares of length $2n$. In particular, the subset of $\ss_2(n)$ found by \cref{theorem:bss} can be detected by a linear-time algorithm, whereas identifying general binary shuffle squares is NP-complete by \cite{bulteau2020}.

Our final result concerns the twins problem. If $s\in \{0,1\}^n$, let $\delta(s)$ be the smallest number of bits that can be removed from $s$ to obtain a shuffle square. Axenovich, Person, and Puzynina \cite{apptwins2012} proved that $\delta(s) = O(n(\log \log n/\log n)^{1/4})$ for all $s\in \{0,1\}^n$ using the regularity method, and we ask a closely related question: how big is $\delta(s)$ for {\it typical} words $s$?

It is not difficult to show that a random binary word $s$ is w.h.p.\footnote{The phrase ``with high probability," abbreviated w.h.p., signifies that an event occurs with probability $1-o(1)$.}\ $\eps$-regular with $\eps = n^{-1/3 + o(1)}$ and thus satisfies $\delta(s) \le n^{2/3 + o(1)}$, using the ideas of \cite{apptwins2012}. Using the greedy algorithm, we prove the following improvement.

\begin{theorem}
\label{theorem:greedytwins}
If $n\ge 1$ and $s$ is a uniform random element of $\{0,1\}^n$, and $h:\NN\rightarrow\RR$ is any function tending to infinity, then w.h.p. we have
\[
\delta(s) \le h(n)\sqrt{n}.
\]
\end{theorem}

Based on numerical evidence (see \cite{mathoverflow} and OEIS A191755 \cite{oeisA191755}), we believe that a much stronger result is true.

\begin{conjecture}
\label{conjecture:half} If $n\ge 1$ and $s$ is a uniform random element of $\{0,1\}^{2n}$ with an even number of ones, then $s$ is w.h.p. a shuffle square.
\end{conjecture}

If true, \cref{conjecture:half} would simultaneously strengthen both \cref{theorem:bss} and \cref{theorem:greedytwins} by showing that $|\ss_2(n)| \ge (1/2 - o(1)) \cdot 4^n$, and that $\delta(s) \le 3$ holds w.h.p. for a random word $s\in\{0,1\}^n$.

The remainder of this paper is organized as follows. In \cref{sec:prelim}, we collect our notations, definitions, and preliminary results, especially standard results about various objects enumerated by Catalan numbers. We relate shuffle squares to standard $2\times n$ Young tableaux and Dyck paths, reverse shuffle squares to $123$-avoiding permutations, and prove two Catalan convolution identities that we need later. In \cref{sec:ss}, we prove \cref{theorem:ss} using inclusion-exclusion. In \cref{sec:rss}, we prove \cref{theorem:rss} using inclusion-exclusion and counting $123$-avoiding permutations of various types. Finally, both Theorems \ref{theorem:greedytwins} and \ref{theorem:bss} are proved in Section \ref{sec:greedy} using a greedy algorithm.

We use the standard asymptotic notation $f=O(g)$ to indicate that there is some constant $C>0$ for which $|f| < Cg$, and add subscripts to $O$ to indicate variables the implicit constant is allowed to depend on.

\vspace{3mm}
\noindent {\bf Acknowledgments.} The authors would like to thank Noga Alon, Ryan Alweiss, and Jacob Fox for stimulating conversations, Matija Buci\'c and Jeffrey Shallit for helpful comments on this paper, and Ben Gunby for a key observation in the proof of \cref{theorem:greedytwins}. We are grateful to the Stanford Undergraduate Research Institute in Mathematics program for hosting this project. The first author's research was supported by NSF GRFP Grant DGE-1656518 and by NSF MSPRF Grant DMS-2103154. 

\section{Preliminaries}\label{sec:prelim}
In this section, we formalize the notions described in the introduction and list the relevant combinatorial identities for the proofs of Theorems \ref{theorem:ss}, \ref{theorem:rss}, \ref{theorem:bss}, and \ref{theorem:greedytwins}.

We also note that an alphabet of size $k$ is conventionally taken to be the set $\{0,1,\dots,k-1\}$, which we abbreviate by $[k]$. For a word $S=s_1s_2 \cdots s_n \in [k]^n$, its \textit{reverse} $S^R$ is the word $s_ns_{n-1} \cdots s_1$.

\begin{definition}
For $k,n\ge 1$, a word $S=s_1s_2 \cdots s_{2n} \in [k]^{2n}$ is a \textit{shuffle square} if and only if there exists index sets $I=\{i_1,i_2,\dots,i_n\}$ and $J=\{j_1,j_2,\dots,j_n\}$ with $i_1 < \cdots < i_n$, $j_1 < \cdots < j_n$, and $I \cap J=\emptyset$ such that $s_{i_r}=s_{j_r}$ for all $1 \le r \le n$.
\end{definition}

\begin{definition}
For $k,n\ge 1$, a word $S=s_1s_2 \cdots s_{2n} \in [k]^{2n}$ is a \textit{reverse shuffle square} if and only if there exists index sets $I=\{i_1,i_2,\dots,i_n\}$ and $J=\{j_1,j_2,\dots,j_n\}$ with $i_1 < \cdots < i_n$, $j_1 < \cdots < j_n$, and $I \cap J=\emptyset$ such that $s_{i_r}=s_{j_{n+1-r}}$ for all $1 \le r \le n$.
\end{definition}

The Catalan numbers are central to the proofs in this paper.
\begin{definition}[Catalan numbers]
The \textit{Catalan numbers} $\{C_n\}$ are defined by $C_0=1$, and 
\[
C_n=\sum_{k=0}^{n-1}C_kC_{n-1-k}
\]
for all $n\ge 1$. 
\end{definition}

It is well-known that $C_n=\frac{1}{n+1}\binom{2n}{n}$ for all nonnegative integers $n$. The proofs in this paper invoke Dyck paths, $123$-avoiding permutations, and standard $2\times n$ Young tableaux, all of which are counted by the Catalan numbers. We define these objects below.

\begin{definition}[UD paths and Dyck paths]\label{def:dyck}
A \textit{UD path of semilength $n$} is a path in the plane starting from $(0,0)$ which consists of $2n$ steps, where each step is either an ``up-step'' of size $(1,1)$ or a ``down-step'' of size $(1,-1)$.
A \textit{Dyck path of semilength $n$} is a UD path of semilength $n$ ending at $(2n,0)$ that never goes below the $x$-axis. A \textit{strict Dyck path} is a Dyck path which does not intersect the $x$-axis internally.
\end{definition}

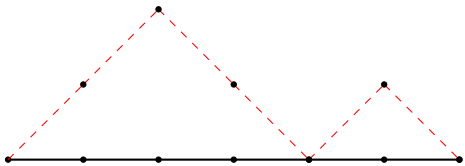
\begin{figure}[H]
    \centering
    \begin{tikzpicture}[scale=0.5]
\draw[thick] (0,0) -- (12,0);

\foreach \x in {0,2,4,6,8,10,12}
    \filldraw[black] (\x,0) circle (2pt);

\draw[dashed, color=red] (0,0) -- (2,2);
\draw[dashed, color=red] (2,2) -- (4,4);
\draw[dashed, color=red] (4,4) -- (6,2);
\draw[dashed, color=red] (6,2) -- (8,0);
\draw[dashed, color=red] (8,0) -- (10,2);
\draw[dashed, color=red] (10,2) -- (12,0);

\filldraw[black] (2,2) circle (2pt);
\filldraw[black] (4,4) circle (2pt);
\filldraw[black] (6,2) circle (2pt);
\filldraw[black] (8,0) circle (2pt);
\filldraw[black] (10,2) circle (2pt);
\filldraw[black] (12,0) circle (2pt);
\end{tikzpicture}
    \caption{A Dyck path of semilength 3}
    \label{fig:dyck path 3}
\end{figure}

UD paths are naturally in bijection with binary words of length $2n$ over the alphabet $\{U, D\}$. Under this bijection, Dyck paths correspond to those words for which every prefix contains at least as many $U$'s as $D$'s. The number of Dyck paths of semilength $n$ is $C_n$, and the number of strict Dyck paths of semilength $n$ is $C_{n-1}$.

\begin{definition}[$123$-avoiding permutation]
\label{def:perm}
Let $\cS_n$ be the set of permutations on $[n]$. A permutation $\pi \in \cS_n$ is called \textit{$123$-avoiding} if there do not exist $i_1<i_2<i_3$ such that $\pi(i_1)<\pi(i_2)<\pi(i_3)$. The family of $123$-avoiding permutations in $\cS_n$ is denoted by $\av_{n}(123)$.
\end{definition}

\begin{figure}[H]
    \centering
    \begin{tikzpicture}
\begin{axis}[
    xlabel={$i$},
    ylabel={$\pi(i)$},
    xmin=1, xmax=8,
    ymin=1, ymax=8,
    xtick={1,2,3,4,5,6,7,8},
    ytick={1,2,3,4,5,6,7,8},
]

\addplot[
    color=red,
    mark=o,
    only marks
    ]
    coordinates {
    (1,5)(2,4)(3,8)(4,2)(5,1)(6,7)(7,6)(8,3)
    };
\end{axis}
\end{tikzpicture}
    \caption{The permutation $\pi = 54821763$ is 123-avoiding.}
\end{figure}
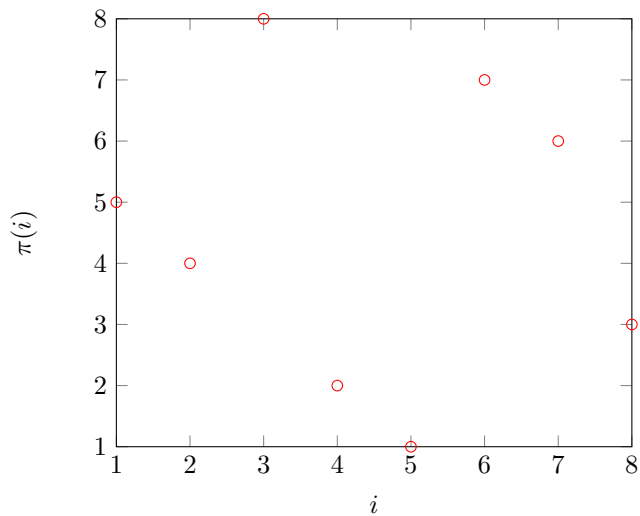

A permutation $\pi$ is $123$-avoiding if and only if $\pi$ can be partitioned into two decreasing subsequences, and the number of $123$-avoiding permutations on $[n]$ is $C_n$.

\begin{definition}[Young diagrams and Young tableaux]
\label{def:tableaux}
A {\it Young diagram} is a finite collection of boxes arranged in left-justified rows, with row lengths in non-increasing order. A {\it standard Young tableau} is obtained from a Young diagram with $n$ boxes by filling in the boxes with the elements of $[n]$, with the requirement that each row and column is increasing.

In this paper, we will work exclusively with standard Young tableaux with two rows (where the second row may be empty) and an even number $2n$ of boxes. We refer to these simply as ``tableaux of semilength $n$.'' For such a tableau we use the notation $(I,J)$ where $I$ and $J$ are the first and second rows (respectively) of the tableau. We say that a tableau of semilength $n$ is {\it rectangular} if it has dimensions $2\times n$.

\begin{figure}[H]
    \centering
    \begin{tikzpicture}[scale=0.5]
\draw[black, thick] (0,0) rectangle (6, 3);
\draw[black, thick] (0, 1.5) -- (6, 1.5);
\draw[black, thick] (2, 0) -- (2, 3);
\draw[black, thick] (4, 0) -- (4, 3);
\node[text width = 1cm] at (1.75, 0.75) {3};
\node[text width = 1cm] at (1.75, 2.25) {1};
\node[text width = 1cm] at (3.75, 0.75) {4};
\node[text width = 1cm] at (3.75, 2.25) {2};
\node[text width = 1cm] at (5.75, 0.75) {6};
\node[text width = 1cm] at (5.75, 2.25) {5};
\end{tikzpicture}
    \caption{A tableau $(I,J) = ((1,2,5), (3,4,6))$ of semilength 3.}
\end{figure}
\end{definition}

The number of rectangular tableaux of semilength $n$ is also $C_n$.

We now state and prove the relevant combinatorial identities on these objects.

A \textit{valley} in a Dyck path is a down-step followed by an up-step. We will require the enumeration of valleys across all Dyck paths of semilength $n$ for our proof of Theorem \ref{theorem:ss}. The enumeration itself is certainly not new (see \cite{deutsch1999} and OEIS A002054 \cite{oeisA002054}), but we include a short proof for completeness.

\begin{proposition}
\label{prop:valleys}
The number of valleys across all Dyck paths of semilength $n$ is $\binom{2n-1}{n+1}$.
\end{proposition}

\begin{proof}
For $n \ge 0$, let $V_n$ be the total number of valleys across all Dyck paths of semilength $n$. We will derive a recursive formula for $V_n$ that can be solved explicitly via generating functions.

For $1 \le k \le n$, let $\cD_{n,k}$ be the set of Dyck paths of semilength $n$ that return to the $x$-axis for the first time at the point $(2k,0)$. We use $\cD_{n}$ to denote Dyck paths whose semilength is $n$. Furthermore, let $V_{n,k}$ be the number of valleys across paths in $\cD_{n,k}$. Each path $p \in \cD_{n,k}$ is of the form $UaDb$, where $a$ is a Dyck path of semilength $k-1$ and $b$ is a Dyck path of semilength $n-k$. Three kinds of valleys appear across these $p$: the valleys in $a\in \cD_{k-1}$ which are each counted $C_{n-k}$ times, the valleys in $b\in \cD_{n-k}$ which are counted $C_{k-1}$ times, and the valley at the point $(2k,0)$, assuming $k<n$, counted $|\cD_{n,k}| = C_{k-1}C_{n-k}$ times.

Thus, we have
\begin{align*}
V_n &=\sum_{k=1}^{n}\left(V_{k-1}C_{n-k}+V_{n-k}C_{k-1}+C_{k-1}C_{n-k}\right) - C_{n-1} \\
&=2\sum_{k=0}^{n-1}V_kC_{n-1-k}+C_n-C_{n-1},
\end{align*}
where we used the Catalan recursion in the last line.

Let $v(x)=\sum_{n=0}^{\infty}V_nx^n$ be the generating function of the sequence $(V_n)_{n\geq0}$. Applying the ``Snake Oil" method described in \cite{wilf1994}, we multiply both sides of the above recursion by $x^n$ and sum over all $n \ge 1$ to obtain $$v(x)=2xv(x)c(x)+(1-x)c(x)-1,$$ where $c(x)$ is the generating function of the Catalan numbers.
Solving for $v(x)$ and plugging in the closed form of $c(x)$ gives $$v(x)=\dfrac{1}{\sqrt{1-4x}}\left(\dfrac{1-\sqrt{1-4x}}{2x}\right)(1-x)-\dfrac{1}{\sqrt{1-4x}}.$$ Expanding out these formal power series and comparing coefficients (see \cite[pages 53-54]{wilf1994}), we obtain
\[
    V_n =\dbinom{2n+1}{n}-\dbinom{2n-1}{n-1}-\dbinom{2n}{n} =\dbinom{2n-1}{n+1},
\]
as desired.
\end{proof}

We conclude this section with a combinatorial identity involving the {\it Catalan convolution} defined by
\[
C_{n,k}\coloneqq\frac{k}{2n-k}\binom{2n-k}{n},
\]
which enumerates (see \cite{connolly2014}) the number of 123-avoiding permutations
$\pi$ of length $n$ with $\pi(k)=n$. The numbers $C_{n,k}$ are called Catalan convolutions because they satisfy
\begin{equation}
C_{n,k}=\sum_{a_{1}+\cdots+a_{k}=n-k}\prod_{i=1}^{k}C_{a_{i}}.\label{eq:conv-def}
\end{equation}

The proof of Theorem \ref{theorem:rss} relies on the following identity involving Catalan convolutions.

\begin{proposition}
\label{prop:catconv}
For all $n \ge 2$,
\begin{equation}
\label{eq:interleaving-identity}
\binom{2n-2}{n-2}=\sum_{a+b+c+d=n-2}\binom{a+c}{a}C_{a+b+1,a+1}C_{c+d+1,c+1}.
\end{equation}
\end{proposition}

In both equations (\ref{eq:conv-def}) and (\ref{eq:interleaving-identity}), the sum is over all nonnegative compositions, i.e. choices of the summands from nonnegative integers.

We first note that $C_{n,k}$ is exactly
the number of Dyck paths from $(0,0)$ to $(2n,0)$ which touch the
$x$-axis exactly $k-1$ times internally; this is because such a
path breaks down into $k$ subpaths of lengths $a_{1}+1,\ldots,a_{k}+1$
which each stay on or above the line $y=1$ internally, hence (\ref{eq:conv-def}).

\begin{proof}[Proof of Proposition \ref{prop:catconv}]
The proof is by double-counting. We claim that both sides enumerate
the family $\cF$ of UD paths between $(0,0)$ and $(2n,0)$ that start
and end with an up-step. The
left side of (\ref{eq:interleaving-identity}) clearly enumerates such paths, because exactly $\text{\ensuremath{\binom{2n-2}{n-2}}}$
words in $\{U, D\}^{2n}$ with $n$
$U$'s and $n$ $D$'s both start and end with $U$.
As for the right side, take any $p\in \cF$ and suppose it intersects the
line $x=0$ a total of $t\ge1$ times internally. These $t$ points
break $p$ up into $t+1\ge2$ segments, each of which is either a
strict Dyck path or else the reflection of a strict Dyck path over the
$x$-axis. Since $p$ starts and ends with an up-step, it must contain at least one segment above the $x$-axis and at least one below it. Let there be $a+1$ segments above the $x$-axis and $b+1$ segments below the $x$-axis. Then, we map $p$ to the pair $(p_{+},p_{-})$
of Dyck paths where $p_{+}$ is obtained by concatenating all the segments above the $x$-axis together, and $p_{-}$ by concatenating all the
segments below the $x$-axis together, and reflect across the $x$-axis.

It is easy to check that this is a surjective map from $\cF$ to the
union 
\[
\bigcup_{a+b+c+d=n-2}D'_{a+b+1,a+1}\times D'_{c+d+1,c+1}, 
\]
where $\cD'_{n,k}$
is the family of Dyck paths of semi-length $n$ with exactly $k-1$
internal points, so that $|\cD'_{n,k}|=C_{n,k}$. Furthermore, the preimage
of any given pair $(p_{+},p_{-})$ has size exactly $\binom{a+c}{a}$, because this
is the number of ways to interleave the $a+1$ segments of $p_{+}$
and the $c+1$ segments of $p_{-}$, excepting the first segment of
$p_{+}$ which must go at the beginning of $p\in F$, and the last
segment of $p_{-}$ which must go at the end. This completes the proof
of (\ref{eq:interleaving-identity}). 
\end{proof}

\section{Shuffle Squares Over Large Alphabets}
\label{sec:ss}
\global\long\def\comp{\textnormal{comp}}%

In this section, we prove Theorem \ref{theorem:ss}, which states that $$\card{\ss_k(n)}=\dfrac{1}{n+1}\dbinom{2n}{n}k^n-\dbinom{2n-1}{n+1}k^{n-1}+O_n(k^{n-2}).$$

If $s\in[k]^{\ell}$ and $I\subseteq[\ell]$, write $s_{I}$
for the subsequence of $s$ indexed by $I$. The connection between tableaux and shuffle squares is the following simple lemma, which was proved by Bukh and Zhou \cite[Lemma 17]{bukh2013}. We include a short proof for completeness.
\begin{lemma}
\label{lemma:index-tableau}If $s\in[k]^{2n}$ is a shuffle square,
then there exists a rectangular tableau $(I,J)$ of semilength $n$ for which $s_I = s_J$. 

\end{lemma}

\begin{proof}
By the definition of a shuffle square, there exists a partition $I' \sqcup J' = [2n]$ for which $s_{I'} = s_{J'}$. If $I' = \{i_1',\ldots, i_n'\}$ and $J' = \{j_1',\ldots, j_n'\}$, then take $i_r = \min (i_r', j_r')$ and $j_r = \max (i_r', j_r')$. It is easy to see that $(I,J) = ((i_1,\ldots, i_n), (j_1,\ldots ,j_n))$ is a tableau for which $s_I = s_J$, as desired.
\end{proof}

Let $\mathcal{T}_n$ denote the family
of all tableaux of semilength $n$. If $s\in[k]^{2n}$ is a shuffle square, we say that $(I,J) \in \mathcal{T}_n$ \emph{is
an index tableau for} $s$ if $s_I = s_J$. Lemma~\ref{lemma:index-tableau} implies the existence
of index tableaux for all shuffle squares. It follows that 
\begin{equation}\label{eq:overcount}
|\ss_{k}(n)|\le|\mathcal{T}_n|\cdot k^{n}=C_{n}\cdot k^{n},
\end{equation}
since this latter expression counts the number of ways to choose an index tableau $(I,J)$ and then the value of $s_{I}$, which together
determine $s$ completely. In order to prove \cref{theorem:ss}, it suffices to understand how much (\ref{eq:overcount}) overcounts.

\begin{proof}[Proof of Theorem \ref{theorem:ss}.]
First, we identify $\mathcal{T}_n$ with the family  $\cM_n$ of \emph{non-nesting perfect matchings} on vertex set $[2n]$. Here, a perfect
matching on $[2n]$ is \emph{non-nesting} if there do not exist
two edges $(i,j)$ and $(i',j')$ satisfying $i<i'<j'<j$. Thus, a
perfect matching $m$ on $[2n]$ is non-nesting if and only if
there exists $(I,J)\in \mathcal{T}_n$ such that the edges in $m$ are exactly the columns in $(I,J)$.

Let $\comp(G)$ denote the number of connected components of a graph
$G$. We claim that
\begin{equation}
|\ss_{k}(n)|=\sum_{m_{1}}k^{\comp(m_{1})}-\sum_{m_{1}\ne m_{2}}k^{\comp(m_{1}\cup m_{2})}+\cdots+(-1)^{r+1}\sum_{m_{1},\ldots,m_{r}}k^{\comp(m_{1}\cup\cdots\cup m_{r})}+\cdots\label{eq:PIE}
\end{equation}
by inclusion-exclusion, where the $r$-th sum is over all unordered choices of $r$ distinct
$m_{i}\in \mathcal{M}_n$. Formula (\ref{eq:PIE}) holds because the number
of shuffle squares $s$ which have the tableaux corresponding to $m_{1},\ldots,m_{r}$ simultaneously
as its index tableaux is $k^{\comp(m_{1}\cup\cdots\cup m_{r})}$,
since the value of $s$ on every vertex of a given connected component
must be the same. But the total number of terms in (\ref{eq:PIE})
is $O_{n}(1)$, because it is at most exponential in the size of $\mathcal{M}_n$, which only depends on $n$. Therefore, to prove Theorem~\ref{theorem:ss} it suffices to select only the terms from (\ref{eq:PIE}) with $\comp(m_{1}\cup\cdots\cup m_{r})\ge n-1$,
as all other terms summed together will be $O_{n}(k^{n-2})$.

It is not hard to see that the only terms in (\ref{eq:PIE}) with
$\comp(m_{1}\cup\cdots\cup m_{r})=n$ are exactly the terms of the
first summation $r=1$, which adds up to $C_{n}\cdot k^{n}$, the
desired leading term. As for $\comp(m_{1}\cup\cdots\cup m_{r})=n-1$,
one can check that the edge union of any three distinct non-nesting matchings on $[2n]$ has at most $n-2$ components, so the only second-order terms appear in the second sum $r=2$. Thus, we need to count
the number of pairs $m_{1}\ne m_{2}$ in $\mathcal{M}_n$ such that $\comp(m_{1}\cup m_{2})=n-1$.

Since $m_{1}$ and $m_{2}$ themselves each have $n$ components (i.e.
edges) of size $2$, for $\comp(m_{1}\cup m_{2})=n-1$ to hold, $m_{1}$
must share all but two of its edges with $m_{2}$, and the two remaining
edges must form a four-cycle with the two corresponding edges of $m_{2}$.
If the vertices of this four-cycle are $a<b<c<d$, then since $m_{1}$
and $m_{2}$ are both non-nesting they cannot contain the edges
$(a,d)$ and $(b,c)$. We may thus assume without loss of generality
that $(a,b),(c,d)\in m_{1}$ and $(a,c),(b,d)\in m_{2}$.

We claim that in order for $\comp(m_{1}\cup m_{2})=n-1$, the four
indices must satisfy the additional property $c=b+1$. If not, there
exists some $x$ between $b$ and $c$, and $x$ is matched to the
same vertex $y$ in both $m_{1}$ and $m_{2}$ since $m_{1}$ and
$m_{2}$ are identical outside $\{a,b,c,d\}$. If $y<a$ or $y>d$,
then $m_{1}$ is not non-nesting, while if $a<y<d$ then $m_{2}$
is not non-nesting. This is a contradiction in all cases, so no
such $x$ can exist and $c=b+1$.

We are now ready to prove that the pairs $\{m_{1},m_{2}\}$ satisfying
$\comp(m_{1}\cup m_{2})=n-1$ are in bijection with pairs $(p,v)$
of a Dyck path $p$ of semilength $n$ and a valley in $p$.
We first define a bijection between Dyck paths and tableaux by taking a path
$p$ to a tableau $(I,J)$ where $I$ consists of the indices of up-steps in $p$ in increasing order, and $J$ consists of the indices of the down-steps in increasing order. A valley $v$ is a down-step followed by an up-step, so under this map it goes to an index $b\in J$ such that $b+1 \in I$. Thus the pairs $(p,v)$ are in bijection with the pairs $((I,J), b)$ where $b\in J$ and $b+1\in I$.

Let $m_1$ be the non-nesting perfect matching corresponding
to $(I,J)$, and let $m_2$ be the matching corresponding to $(I\cup\{b\}\backslash\{b+1\},J\cup\{b+1\}\backslash\{b\})$ \footnote{Here we abuse notation by conflating sorted tuples $I$ and $J$ with their underlying sets; thus $I \cup \{b\} \setminus \{b+1\}$ is the tuple obtained by replacing the occurrence of $b$ in $I$ with $b+1$.}.
It is easy to see that $\comp(m_{1}\cup m_{2})=n-1$, and this gives
a bijection between pairs $((I,J),b)$ where $b\in J$ and $b+1 \in I$, and pairs $(m_{1},m_{2})$
with $\comp(m_{1}\cup m_{2})=n-1$.

By \cref{prop:valleys}, the
number of valleys across all Dyck paths of semilength $n$ is $\binom{2n-1}{n+1}$.
Thus, this is also the number of terms in (\ref{eq:PIE}) equal to
$-k^{n-1}$. We find that
\[
|\ss_{k}(n)|=C_{n}k^{n}-\binom{2n-1}{n+1}k^{n-1}+O_n(k^{n-2}),
\]
completing the proof.
\end{proof}

\section{Reverse Shuffle Squares}
\label{sec:rss}
In this section, we prove Theorem \ref{theorem:rss}, which states that $|\rss_{k}(n)|=C_n k^{n}-B_n k^{n-1}+O_{n}(k^{n-2}),$ where
\[
B_n = 2\binom{2n-2}{n-2}+2C_{n+1}-8C_{n}+5C_{n-1} = \dfrac{2n^3+9n^2-35n+30}{n^3+3n^2+2n}\dbinom{2n-2}{n-1},
\]
for $n\ge 2$. This time, instead of interpreting the Catalan numbers in terms of
Dyck paths, we will interpret them in terms of $123$-avoiding permutations.

Given a permutation $\pi\in S_{n}$
and a word $s\in[k]^{n}$, we write $\pi(s)$ for the word obtained
by shuffling the letters according to $\pi$, i.e. $\pi(s)_{i}\coloneqq s_{\pi(i)}$.
We also write $s^{R}$ for the reverse of $s$.
\begin{lemma}
Suppose $s\in[k]^{2n}$ and we split $s=s's''$ into two equal halves,
so that $s',s''$ are both words in $[k]^{n}$. Then, $s$ is a reverse
shuffle square if and only if $s''=\pi(s')$ for some $\pi \in \av_n(123)$.
\end{lemma}

\begin{proof}
We first prove the only-if direction in the special case that $k\ge n$
and every letter in $s$ appears exactly twice.

It was shown in \cite{henshall2012} that if $s$ is a reverse shuffle square,
then $s$ is an \emph{abelian square}, which is a word where the second
half is a permutation of the first. Thus, $s''=\pi(s')$ for some
permutation $\pi$. Since every letter in $s$ appears exactly twice,
this $\pi$ is unique. We show that it is $123$-avoiding. If not,
there are three indices $i_{1}<i_{2}<i_{3}$ for which $\pi(i_{1})<\pi(i_{2})<\pi(i_{3})$.
Thus, $s'_{i_{1}},s'_{i_{2}},s'_{i_{3}}$ appear in the same relative
order in $s'$ as they do in $s''$. These six letters appear at positions
$i_{1}<i_{2}<i_{3}<n+\pi(i_{1})<n+\pi(i_{2})<n+\pi(i_{3})$ in the
original word $s$.

Since $s$ is a reverse shuffle square, its restriction to
the six positions above is a reverse shuffle square as well, as the three letters there do not appear
elsewhere in $s$. But the restriction to these six positions of $s$
is a word of the form $abcabc$, which cannot be a reverse shuffle
square. This proves the special case.

For the general case, suppose $t\in[k]^{2n}$ is any reverse shuffle
square, which means that there exists a partition $I\sqcup J = [2n]$ such that $t_I = t_J^R$. Define $s\in[n]^{2n}$ so that $s_{I}=1...n$ and $s_{J}=n...1$,
so that $s$ is a reverse shuffle square where every letter appears exactly
twice. By the definition of $I$ and $J$ in terms of $t$, $t$ is a homomorphic image of $s$ (in other words, there is a way to obtain $t$ by replacing the letters in $s$ by the letters in $[k]$). Applying the claim above to $s$, we obtain a $123$-avoiding
permutation $\pi$ such that the second half of $s$ is $\pi$ applied
to the first half. As $t$ is a homomorphic image of $s$, this holds
for $t$ as well with the same $\pi$, which proves the only-if direction.

To prove the if direction, note that a permutation $\pi$ is $123$-avoiding
if and only if it can be partitioned into two decreasing subsequences.
Suppose $s$ satisfies $s''=\pi(s')$ for such a $\pi$, and let $[n]=I_{\pi}\sqcup J_{\pi}$
be a partition of the index set of $\pi$ for which $\pi|_{I_{\pi}}$
and $\pi|_{J_{\pi}}$ are both decreasing. Define $I\coloneqq I_{\pi}\cup(n+\pi(J_{\pi}))$
and $J\coloneqq J_{\pi}\cup(n+\pi(I_{\pi}))$, we see that $I$ and
$J$ partition $[2n]$. Because $\pi$ is decreasing when restricted
to both $I_{\pi}$ and $J_{\pi}$, it follows that the part of $s_{I}$
in $s''$ is the reverse of the part of $s_{J}$ in $s'$, and similarly
the part of $s_{J}$ in $s''$ is the reverse of the part of $s_{I}$
in $s'$. This means that $s_{I}=s_{J}^{R}$, completing the proof
that $s$ is a reverse shuffle square.
\end{proof}
We obtain an upper bound $|\rss_{k}(n)|\le C_{n}k^{n}$
by sending each reverse shuffle square $s$ to an ordered pair $(\pi,s')$
of a $123$-avoiding permutation $\pi$ corresponding to $s$ and
the first half $s'$ of $s$. The full word $s$ can be reconstructed
from this data by taking $s''=\pi(s')$. It remains to understand
the overcounting to get at the second-order term.  

To each $\pi\in\av_{n}(123)$, associate the matching $m(\pi)$ on
$[2n]$ whose edges are $(i,n+\pi(i))$. Recall that we defined $\comp(G)$ to be the number of connected components in a graph $G$, and that for $1\le i \ne j \le n$, the transposition $(ij)$ is the permutation which swaps $i$ and $j$ and fixes every other integer from $1$ to $n$.

\begin{lemma} For $n, k \ge 2$, we have
\[
|\rss_{k}(n)|=C_{n}k^{n}-B_{n}k^{n-1}+O_{n}(k^{n-2}),
\]
where $B_{n}$ is the number of unordered pairs $\pi_{1},\pi_{2}\in\av_{n}(123)$
such that $\pi_1 = \pi_2 \circ (ij)$ for some transposition $(ij)\in \av_n(123)$.
\end{lemma}
\begin{proof}

We define $S_{\pi}$ to be
the set of $k^{n}$ words of the form $s=s'\pi(s')$ in $[k]^{2n}$.
We obtain that $s\in S_{\pi}$ exactly if $s_{i}=s_{j}$ whenever
$i\sim j$ in $m(\pi)$. As a result, for multiple permutations $\pi_{1},\cdots,\pi_{r}$,
the intersection $S_{\pi_{1}}\cap\cdots\cap S_{\pi_{r}}$ is exactly
the set of words $s\in[k]^{2n}$ which are constant on every connected
component of $m(\pi_{1})\cup m(\pi_{2})\cup\cdots\cup m(\pi_{r})$.
By inclusion-exclusion, we obtain
\[
|\rss_{k}(n)|=\sum_{\pi}k^{n}-\sum_{\pi_{1},\pi_{2}}k^{\comp(m(\pi_{1})\cup m(\pi_{2}))}+\cdots+(-1)^{r+1}\sum_{\pi_{1},\ldots,\pi_{r}}k^{\comp(m(\pi_{1})\cup\cdots\cup m(\pi_{r}))}+\cdots,
\]
where the $r$-th sum is a sum over unordered $r$-tuples of distinct
$\pi_{i}\in\av_{n}(123)$. We will show that all terms of the form $k^{n}$ appear
in the first sum, and that all terms of the form $k^{n-1}$ appear in the second. Observe that for all $\pi\in\av_n(123)$, $m(\pi)$ is \textit{precedence-free}
(doesn't include two edges $(i_{1},j_{1}),(i_{2},j_{2})$ with $i_{1}<j_{1}<i_{2}<j_{2}$).

We observe that all terms of the form $k^n$ appear in the first sum. Indeed, in any graph $m(\pi_1) \cup \cdots \cup m(\pi_r)$ with $r \geq 2$, there exist edges $(a, b) \in m(\pi_1)$ and $(a, d) \in m(\pi_r)$ with $b \neq d$, so $\comp(m(\pi_1) \cup \cdots \cup m(\pi_r)) \leq n-1$. 

Next, we claim that all terms of the form $k^{n-1}$ appear in the second sum. In other words, we claim that if $r\ge 3$, the graph $m(\pi_1) \cup \cdots \cup m(\pi_r)$ has at most $n-2$ connected components. If $\comp(m(\pi_1) \cup m(\pi_2)) \leq n-2$, then we are done. Otherwise, suppose $\comp(m(\pi_1) \cup m(\pi_2)) = n-1$, so that there is a unique component of size $4$ on vertices $a<b<c<d$. In order for $m(\pi_3)$ to not be identical to $m(\pi_1)$ and $m(\pi_2)$ and also for $\comp(m(\pi_1) \cup m(\pi_2)\cup m(\pi_3)) = n-1$, each of the three perfect matchings in $\{a,b,c,d\}$ must appear in one of $m(\pi_1),m(\pi_2),$ and $m(\pi_3)$. But then one of these matchings contains the edges $(a,b)$ and $(c,d)$, and cannot be precedence-free. This proves the claimed formula with $B_n$ counting the number of unordered pairs of $\pi_1,\pi_2 \in \av_n(123)$ with $\comp(m(\pi_{1})\cup m(\pi_{2}))=n-1$.

The only way for
$\comp(m(\pi_{1})\cup m(\pi_{2}))=n-1$ to occur is if $\pi_{1}$
and $\pi_{2}$ differ by exactly one transposition (i.e. $\pi_{1}=\pi_{2}\circ (ij)$
in cycle notation for some $i,j\in[n]$), so that $m(\pi_{1})\cup m(\pi_{2})$
has exactly one component of size $4$. This completes the proof.
\end{proof}

Thus, to prove \cref{theorem:rss}, it remains to show 
\begin{equation}
B_{n}=2\binom{2n-2}{n-2}+2C_{n+1}-8C_{n}+5C_{n-1}\label{eq:bn-closed-form}
\end{equation}
for $n\ge2$. This we do below.

\subsection{A Closed Form for $B_n$}
In this section, we prove the following formula for $B_{n}$, which
is defined for $n\ge1$ as the number of unordered pairs of elements
of $\av_{n}(123)$ which differ by a single transposition, which is
almost all the way towards (\ref{eq:bn-closed-form}).
\begin{lemma}
\label{lem:pairs-closed-form}For all $n\ge2$,
\[
B_{n}=2\binom{2n-2}{n-2}+2C_{n+1}-8C_{n}+5C_{n-1}.
\]
\end{lemma}
This lemma would complete the proof of (\ref{eq:bn-closed-form}) and thus \cref{theorem:rss}. \cref{lem:pairs-closed-form} will follow from another application of inclusion-exclusion,
which depends on the following diagrams.

\vspace{-0.3cm}

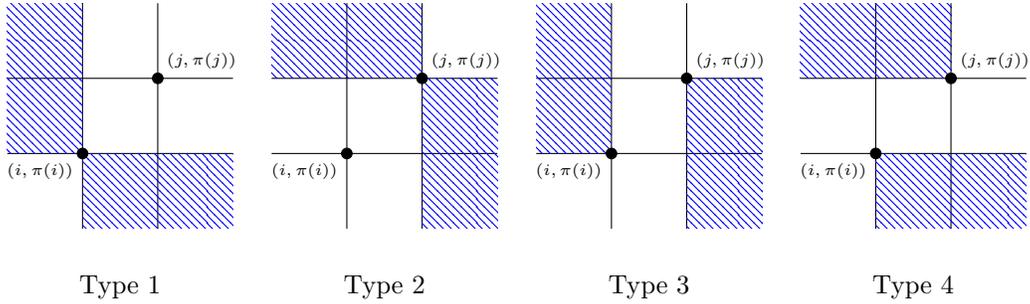
\begin{figure}[H]
    \centering
    \caption{Allowed regions in each type of permutation.}
    \medskip
    \begin{tikzpicture}
\draw (0,0) -- (3,0);
\draw (0,1) -- (3,1);
\draw (1,2) -- (1,-1);
\draw (2,2) -- (2,-1);

\fill[pattern=north west lines, pattern color=blue] (0,0) rectangle (1,1);
\fill[pattern=north west lines, pattern color=blue] (0,1) rectangle (1,2);
\fill[pattern=north west lines, pattern color=blue] (1,-1) rectangle (2,0);
\fill[pattern=north west lines, pattern color=blue] (2,-1) rectangle (3,0);

\filldraw[black] (1,0) circle (2pt);
\filldraw[black] (2,1) circle (2pt);

\node[anchor=north east] at (1,0){\tiny $(i, \pi(i))$};
\node[anchor=south west] at (2,1){\tiny $(j, \pi(j))$};

\node[anchor=north] at (1.5,-1.5){Type 1};
\end{tikzpicture}
    \begin{tikzpicture}
\draw (0,0) -- (3,0);
\draw (0,1) -- (3,1);
\draw (1,2) -- (1,-1);
\draw (2,2) -- (2,-1);

\fill[pattern=north west lines, pattern color=blue] (1,1) rectangle (2,2);
\fill[pattern=north west lines, pattern color=blue] (0,1) rectangle (1,2);
\fill[pattern=north west lines, pattern color=blue] (2,0) rectangle (3,1);
\fill[pattern=north west lines, pattern color=blue] (2,-1) rectangle (3,0);

\filldraw[black] (1,0) circle (2pt);
\filldraw[black] (2,1) circle (2pt);

\node[anchor=north east] at (1,0){\tiny $(i, \pi(i))$};
\node[anchor=south west] at (2,1){\tiny $(j, \pi(j))$};

\node[anchor=north] at (1.5,-1.5){Type 2};
\end{tikzpicture}
    \begin{tikzpicture}
\draw (0,0) -- (3,0);
\draw (0,1) -- (3,1);
\draw (1,2) -- (1,-1);
\draw (2,2) -- (2,-1);

\fill[pattern=north west lines, pattern color=blue] (0,0) rectangle (1,1);
\fill[pattern=north west lines, pattern color=blue] (0,1) rectangle (1,2);
\fill[pattern=north west lines, pattern color=blue] (2,0) rectangle (3,1);
\fill[pattern=north west lines, pattern color=blue] (2,-1) rectangle (3,0);

\filldraw[black] (1,0) circle (2pt);
\filldraw[black] (2,1) circle (2pt);

\node[anchor=north east] at (1,0){\tiny $(i, \pi(i))$};
\node[anchor=south west] at (2,1){\tiny $(j, \pi(j))$};

\node[anchor=north] at (1.5,-1.5){Type 3};
\end{tikzpicture}
    \begin{tikzpicture}
\draw (0,0) -- (3,0);
\draw (0,1) -- (3,1);
\draw (1,2) -- (1,-1);
\draw (2,2) -- (2,-1);

\fill[pattern=north west lines, pattern color=blue] (1,1) rectangle (2,2);
\fill[pattern=north west lines, pattern color=blue] (0,1) rectangle (1,2);
\fill[pattern=north west lines, pattern color=blue] (1,-1) rectangle (2,0);
\fill[pattern=north west lines, pattern color=blue] (2,-1) rectangle (3,0);

\filldraw[black] (1,0) circle (2pt);
\filldraw[black] (2,1) circle (2pt);

\node[anchor=north east] at (1,0){\tiny $(i, \pi(i))$};
\node[anchor=south west] at (2,1){\tiny $(j, \pi(j))$};

\node[anchor=north] at (1.5,-1.5){Type 4};
\end{tikzpicture}
    \label{types}
\end{figure}

Recall that every permutation $\pi$ can be represented in the plane
by plotting all the points $(i,\pi(i))$, and $\pi$ is $123$-avoiding
if and only if the plot doesn't contain three points in increasing
order. Suppose $\pi\in\av_{n}(123)$ and there is a transposition
$(ij)$ for which $\pi\circ(ij)\in\av_{n}(123)$ as well. By swapping
$\pi$ with $\pi\circ(ij)$ if necessary, we may assume $\pi(i)<\pi(j)$
as in the diagram. Then, the four vertical and horizontal lines through
the two points $(i,\pi(i))$ and $(j,\pi(j))$ divide the plane into
nine rectangular sectors, as shown. We say that the pair $(\pi,(ij))$
is of \emph{type t} (for $t\in[4]$) if all the remaining points in
the plot of $\pi$ fall into only the shaded regions in the picture
labelled ``Type $t$.'' For example, $(\pi,(ij))$ is of type $1$
if and only if for all $i'\not\in\{i,j\}$, either $i'<i$ and $\pi(i')> \pi(i)$,
or $i'> i$ and $\pi(i')< \pi(i)$. Note that it's possible for a pair to
be of more than one type.

\begin{figure}[H]
    \centering
    \begin{tikzpicture}
\draw (0,0) -- (3,0);
\draw (0,1) -- (3,1);
\draw (1,2) -- (1,-1);
\draw (2,2) -- (2,-1);

\node at (0.5,1.5){$s_{0,2}$};
\node at (1.5,1.5){$s_{1,2}$};
\node at (2.5,1.5){$s_{2,2}$};

\node at (0.5,0.5){$s_{0,1}$};
\node at (1.5,0.5){$s_{1,1}$};
\node at (2.5,0.5){$s_{2,1}$};

\node at (0.5,-0.5){$s_{0,0}$};
\node at (1.5,-0.5){$s_{1,0}$};
\node at (2.5,-0.5){$s_{2,0}$};
\end{tikzpicture}
    \caption{The axis-parallel lines through $(i, \pi(i))$ and $(j, \pi(j))$ partition the plot of $\pi$ into $9$ cells.}
    \label{fig:sectors}
\end{figure}

\begin{lemma}
If $\pi\in\av_{n}(123)$, $1\le i<j\le n$, and $\pi(i)<\pi(j)$,
and $\pi\circ(ij)\in\av_{n}(123)$, then $(\pi,(ij))$ belongs to (at least)
one of the four types.
\end{lemma}

\begin{proof}
Label the nine sectors as $s_{x,y}$ as in Figure \ref{fig:sectors}, so that $x=0$ if the sector is left of $i$, $x=1$ if it
is between $i$ and $j$, and $x=2$ if it is to the right of $y$,
and similarly for $y$. Since $\pi\in\av_{n}(123)$, $s_{0,0},s_{1,1}$
and $s_{2,2}$ must be empty, since any point in any of them would
form a $123$-pattern with $\pi(i)$ and $\pi(j)$. Thus these three
sectors are always empty, as in the diagram.

Next, note that $s_{0,1}$ and $s_{1,2}$ cannot both be nonempty,
since a point in $s_{0,1}$ and a point in $s_{1,2}$ would form a
$123$-pattern with $(i,\pi(j))$ in $\pi\circ(ij)$. Similarly, at
least one of $s_{1,0}$ and $s_{2,1}$ may be nonempty if $(j,\pi(i))$
appears in the diagram for $\pi\circ(ij)$. This completes the proof.
\end{proof}

Let $P_{n,t}$ denote the collection of pairs $(\pi,(ij))$ of $\pi\in\av_{n}(123)$
and $1\le i<j\le n$ for which $1\le i<j\le n$ of type $t$ for $t=1,2,3,4$.
Clearly, $\cup_{t=1}^{4}P_{n,t}$ is in bijection with the set of
pairs $\{\pi_{1},\pi_{2}\}\in\binom{\av_{n}(123)}{2}$ differing by
a transposition, so it suffices to enumerate this union. We proceed
by inclusion-exclusion.
\begin{lemma}
\label{lem:p-intersections}For $n\ge2$, collections $P_{n,t}$ satisfy
\begin{align}
|P_{n,1}|=|P_{n,2}| & =C_{n+1}-2C_{n}, \label{eq:pn1} \\
|P_{n,3}|=|P_{n,4}| & =\binom{2n-2}{n-2}, \label{eq:pn2}\\
|P_{n,1}\cap P_{n,2}|=|P_{n,3}\cap P_{n,4}| & =C_{n-1},\label{eq:pn3}\\
|P_{n,1}\cap P_{n,3}|=|P_{n,1}\cap P_{n,4}|=|P_{n,2}\cap P_{n,3}|=|P_{n,2}\cap P_{n,4}| & =C_{n}-C_{n-1},\label{eq:pn4}\\
|P_{n,1}\cap P_{n,2}\cap P_{n,3}| 
=|P_{n,1}\cap P_{n,2}\cap P_{n,4}| \nonumber \\
=|P_{n,1}\cap P_{n,3}\cap P_{n,4}|
=|P_{n,2}\cap P_{n,3}\cap P_{n,4}| 
& =C_{n-1},\label{eq:pn5}\\
|P_{n,1}\cap P_{n,2}\cap P_{n,3}\cap P_{n,4}| & =C_{n-1}.\label{eq:pn6}
\end{align}
\end{lemma}

Before we prove the lemma, note that it implies Lemma \ref{lem:pairs-closed-form}
by inclusion-exclusion. Indeed, we have
\begin{align*}
B_{n}&=\Big|\bigcup_{t=1}^{4}P_{n,t}|=[2(C_{n+1}-2C_{n})+2A_{n}]-[2C_{n-1}+4(C_{n}-C_{n-1})]+[4C_{n-1}]-[C_{n-1}] \\
&=2A_{n}+2C_{n+1}-8C_{n}+5C_{n-1}
\end{align*}
by inclusion-exclusion and reading off the values from Lemma \ref{lem:p-intersections}.
\begin{proof}[Proof of Lemma \ref{lem:p-intersections}.]
The system of equations can really be reduced to the four distinct cases arising from (\ref{eq:pn1}), (\ref{eq:pn2}), (\ref{eq:pn3}), and (\ref{eq:pn4}):
$P_{n,1}$, $P_{n,4}$, $P_{n,1}\cap P_{n,2}$ (with only two allowed
regions), and $P_{n,1}\cap P_{n,3}$ (with only three allowed regions). Each of the other cases is equivalent to one of these four via either a $180^\circ$ rotation about the point $\left(\frac{n+1}{2},\frac{n+1}{2}\right)$ (which we denote by $R$) or reflection over the line $y=x$ (which we call $F$).

Now, to see why the latter two transformations preserve $123$-avoiding permutations, consider any $\pi \in \av_n(123)$. Note $R$ maps the point $(i,\pi(i))$ to $(n+1-i,n+1-\pi(i))$. If $R(\pi) \notin \av_n(123)$, then there exist $i>j>k$ such that $n+1-\pi(i)<n+1-\pi(j)<n+1-\pi(k)$. But then $\pi(i)>\pi(j)>\pi(k)$, so $\pi \notin \av_n(123)$, a contradiction.

On the other hand, note that $F(\pi)=\pi^{-1}$, and it is trivial to check that $\sigma^{-1}$ must be $123$-avoiding. This shows that both transformations (a) and (b) preserve $123$-avoiding permutations.

It is easy to see that the sets in (\ref{eq:pn5}) and (\ref{eq:pn6}) define the same regions as $P_{n,1} \cap P_{n,2}$, and that $P_{n,3} \cap P_{n,4}$ is equivalent to $P_{n,1} \cap P_{n,2}$. Moreover, $R$ maps $P_{n,2}$ to $P_{n,1}$ and $P_{n,2} \cap P_{n,3}$ to $P_{n,1} \cap P_{n,3}$. Finally, $F$ maps $P_{n,1} \cap P_{n,4}$ and $P_{n,2} \cap P_{n,4}$ to $P_{n,1} \cap P_{n,3}$. This proves that we only need to consider the four cases outlined above.

We save $P_{n,4}$ to the end, and handle the other three that can immediately be represented in terms of Catalan numbers.\\

\noindent \textbf{Enumeration of $P_{n,1}$}. We start by
proving (\ref{eq:pn1}), which will follow from
\begin{equation}
|P_{n,1}|=\sum_{i=1}^{n-1}C_{i}C_{n-i}.\label{eq:conv1}
\end{equation}
 The proof is by bijection: take two nonempty $123$-avoiding permutations
$\sigma$ and $\tau$ with $|\sigma|+|\tau|=n$. Let $i=|\sigma|$,
and $\pi(i)=n-i=|\tau|$. Given $(\sigma,\tau)$, we obtain $\text{(\ensuremath{\pi,(ij))}}$
of type $1$ as follows.

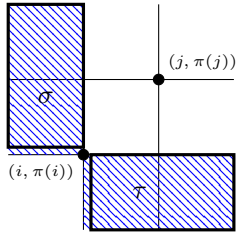
\begin{figure}[H]
    \centering
    \begin{tikzpicture}
\draw (0,0) -- (3,0);
\draw (0,1) -- (3,1);
\draw (1,2) -- (1,-1);
\draw (2,2) -- (2,-1);

\fill[pattern=north west lines, pattern color=blue] (0,0) rectangle (1,2);
\fill[pattern=north west lines, pattern color=blue] (1,-1) rectangle (3,0);

\filldraw[black] (1,0) circle (2pt);
\filldraw[black] (2,1) circle (2pt);

\node[anchor=north east] at (1,0){\tiny $(i, \pi(i))$};
\node[anchor=south west] at (2,1){\tiny $(j, \pi(j))$};

\node at (0.5,0.76){$\mathbf{\sigma}$};
\node at (1.75,-0.5){$\mathbf{\tau}$};

\draw[very thick] (0,0.1) rectangle (1,2);
\draw[very thick] (1.1,-1) rectangle (3,0);
\end{tikzpicture}
    \caption{Constructing permutations of Type $1$ out of two smaller permutations.}
\end{figure}

Place a copy of $\sigma$ in upper left rectangle $[1,i] \times [n-i+1,n]$, and a copy of $\tau$ in the lower-right rectangle $[i+1,n] \times [i+1,n]$, and insert the point $(i,n-i)$. As there are $n+1$ points
in total now, this is not a valid permutation. The offending points
are those in $\sigma$ and $\tau$ which get placed on the horizontal
and vertical lines through $(i,\pi(i))$. Define $(j,\pi(j))$ such
that $j$ is the $x$-coordinate of the offending point in $\tau$,
and $\pi(j)$ is the $y$-coordinate of the offending point in $\sigma$.
Remove the two offending points and insert $(j,\pi(j))$ to obtain
an honest permutation $\pi\in\av_{n}(123)$.

This exhibits a bijection between $P_{n,1}$ and ordered pairs $(\sigma,\tau)$
of nonempty $123$-avoiding permutations whose lengths sum to $n$,
thus proving the convolution formula (\ref{eq:conv1}). This implies
(\ref{eq:pn1}) by the standard convolution identity $C_{n+1}=\sum_{i=0}^{n}C_{i}C_{n-i}$.\\

\noindent \textbf{Enumeration of $P_{n,1} \cap P_{n,2}$}. We now prove the identity $|P_{n,1}\cap P_{n,2}|=C_{n-1}$ using a very similar bijection argument. First form a bijection between $P_{n,1} \cap P_{n,2}$ and ordered pairs of $123$-avoiding permutations $(\sigma,\tau)$ with $\card{\sigma}+\card{\tau}=n-2$. This identity concerns the following regions that are the intersection of Types 1 and 2 in Figure \ref{types}.

\begin{figure}[H]
    \centering
    \begin{tikzpicture}
\draw (0,0) -- (3,0);
\draw (0,1) -- (3,1);
\draw (1,2) -- (1,-1);
\draw (2,2) -- (2,-1);

\fill[pattern=north west lines, pattern color=blue] (0,1) rectangle (1,2);
\fill[pattern=north west lines, pattern color=blue] (2,-1) rectangle (3,0);

\filldraw[black] (1,0) circle (2pt);
\filldraw[black] (2,1) circle (2pt);

\node[anchor=north east] at (1,0){\tiny $(i, \pi(i))$};
\node[anchor=south west] at (2,1){\tiny $(i+1, \pi(i)+1)$};

\node at (0.5,1.5){$\mathbf{\sigma}$};
\node at (2.5,-0.5){$\mathbf{\tau}$};

\draw[very thick] (0,1.1) rectangle (0.9,2);
\draw[very thick] (2.1,-1) rectangle (3,-0.1);
\end{tikzpicture}
    \caption{Constructing permutations that are both Type 1 and Type 2 out of two smaller permutations.}
    \label{fig:p12}
\end{figure}
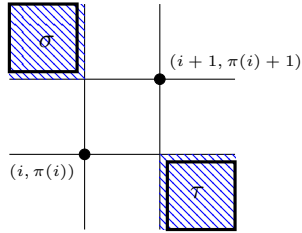

As shown in Figure~\ref{fig:p12}, we want to count the number of nonempty 123-avoiding permutations that have points in only the top left and bottom right regions. Hence, given $(\sigma,\tau)$, we obtain $(\pi,(ij)) \in P_{n,1} \cap P_{n,2}$ by placing a copy of $\sigma$ in the upper-left rectangle $[1,i-1] \times [\pi(i)+2,n]$ of the $n \times n$ grid and a copy of $\tau$ in the lower right rectangle $[n-i+2,n] \times [1,n-i-1]$. Given these regions, we see that $i$ and $j$ are adjacent, as well as $\pi(i)$ and $\pi(j)$, in the permutations we want to count. That is, $j=i+1$ and $\pi(j)=\pi(i)+1$.

Then, the top left region in the diagram can be described as the set of points left of $(i,\pi(i))$ and as the set of points above $(i+1, \pi(i)+1)$. But in any permutation $\pi$, there are $i-1$ points to the left of $i$, and $n - \pi(i)-1$ above $(i+1,\pi(i)+1)$. We see that $i-1 = n - \pi(i)-1$, which determines $\pi(i)=n-i$. Hence, $i-1=\card{\sigma}$ and $n-i-1=\card{\tau}$. Fill in the remainder of the permutation $\pi$ by adding the points $(i,n-i)$ and $(i+1,n-i+1)$.

Since we are free to choose $\sigma \in \av_{i-1}(123)$ and $\tau \in \av_{n-i-1}(123)$, we have $$\card{P_{n,1} \cap P_{n,2}}=\sum_{i=1}^{n-1}C_{i-1}C_{n-i-1}=C_{n-1},$$ as desired, proving (\ref{eq:pn3}).\\

\noindent \textbf{Enumeration of $P_{n,1} \cap P_{n,3}$}. Next, we move on to (\ref{eq:pn4}). Once again, the argument is similar to the one to find $\card{P_{n,1}}$, involving the following diagram illustrating the general form of a permutation $\pi$ belonging to both Types 1 and 3.

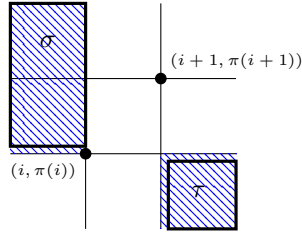
\begin{figure}[H]
    \centering
    \begin{tikzpicture}
\draw (0,0) -- (3,0);
\draw (0,1) -- (3,1);
\draw (1,2) -- (1,-1);
\draw (2,2) -- (2,-1);

\fill[pattern=north west lines, pattern color=blue] (0,0) rectangle (1,2);
\fill[pattern=north west lines, pattern color=blue] (2,-1) rectangle (3,0);

\filldraw[black] (1,0) circle (2pt);
\filldraw[black] (2,1) circle (2pt);

\node[anchor=north east] at (1,0){\tiny $(i, \pi(i))$};
\node[anchor=south west] at (2,1){\tiny $(i+1, \pi(i+1))$};

\node at (0.5,1.5){$\mathbf{\sigma}$};
\node at (2.5,-0.5){$\mathbf{\tau}$};

\draw[very thick] (0,0.1) rectangle (1,2);
\draw[very thick] (2.1,-1) rectangle (3,-0.1);
\end{tikzpicture}
    \caption{Constructing permutations that are both Type 1 and Type 3 out of two smaller permutations.}
\end{figure}

The only points in $\pi$ are located in the rectangles $[1,i] \times [\pi(i), n]$ and $[j,n] \times [1,\pi(i)]$. As there must be no points between $i$ and $j$ horizontally, we have $j=i+1$. Moreover, since there must be $n-i-1$
points below $\pi(i)$ and $i$
points above $\pi(i)$, we have $\pi(i)=n-i$.

This time, we construct a bijection from pairs $(\pi,(ij)) \in P_{n,1} \cap P_{n,3}$ to ordered pairs $(\sigma,\tau)$ with $\sigma \in \av_i(123)$ and $\tau \in \av_{n-i-1}(123)$. Given a pair $(\sigma,\tau)$, we place a copy of $\sigma$ in the upper left rectangle $[1,i] \times [n-i+1, n]$ and a copy of $\tau$ in the lower right rectangle $[i+2,n] \times [1,n-i-1]$, then add the point $(i,n-i)$. To finish, remove the point
$(i,\sigma(i))$
and add the point $(i+1,\sigma(i))$. This creates a valid permutation $\pi$, so $$\card{P_{n,1} \cap P_{n,3}}=\sum_{i=1}^{n-1}C_iC_{n-i-1}=C_n-C_{n-1},$$ proving (\ref{eq:pn3}).\\

\noindent \textbf{Enumeration of $P_{n,4}$}. This leaves
only (\ref{eq:pn2}), which expands as
\[
|P_{n,4}|=\sum_{a+b+c+d=n-2}\binom{a+c}{a}C_{a+b+1,a+1}C_{c+d+1,c+1},
\]
by \cref{prop:catconv}.
This is again a bijection argument, illustrated by the diagram below.

\begin{figure}[H]
    \centering
    \begin{tikzpicture}[scale=0.5]
\draw[color=blue] (0,0) rectangle (4,4);
\filldraw[black] (3,0) circle (2pt);

\draw[dashed] (3,4) -- (3,-4);

\draw[thick, decorate, decoration={calligraphic brace, raise=15pt}]
  (0,4) -- node[above=16pt] {$\sigma$} (5.5,4);

\draw[thick, decorate, decoration={calligraphic brace, mirror, raise=5pt}]
  (0,0) -- node[below=6pt] {$b$} (3,0);
  
\draw[color=blue] (4.5,0) rectangle (4.75,4);
\draw[color=red] (4,-4) rectangle (4.5, -0.25);
\draw[color=red] (4.75,-4) rectangle (5, -0.25);
\draw[color=blue] (5,0) rectangle (5.5,4);

\draw[thick, decorate, decoration={calligraphic brace, mirror, raise=5pt}]
  (4,-4) -- node[below=6pt] {$c$} (5,-4);
  
\draw[thick, decorate, decoration={calligraphic brace, raise=5pt}]
  (3,4) -- node[above=6pt] {$a$} (5.5,4);
  
\draw[color=red] (5.5,-4) rectangle (9.5,-0.25);

\draw[thick, decorate, decoration={calligraphic brace, raise=5pt}]
  (5.5,-0.25) -- node[above=6pt] {$d$} (9.5,-0.25);
  
\draw[thick, decorate, decoration={calligraphic brace, mirror, raise=15pt}]
  (4,-4) -- node[below=16pt] {$\tau$} (9.5,-4);
  
\filldraw[black] (5.5,-0.25) circle (2pt);
\end{tikzpicture}
    \caption{Constructing permutations of Type 4 out of two smaller permutations.}
\end{figure}
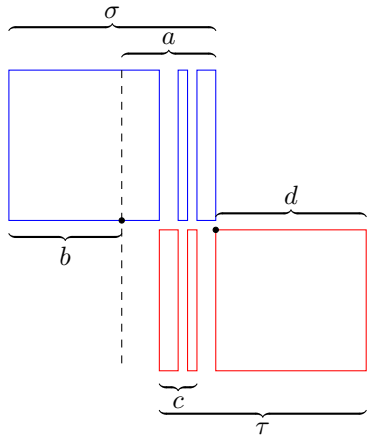

We construct pairs $(\pi, (ij)) \in P_{n,4}$ out of ordered pairs $(\sigma, \tau)\in \av_{a+b+1}(123) \times \av_{c+d+1}(123)$ such that $\sigma(b+1)=1$ and $\tau(c+1)=c+d+1$.  The number of such pairs is exactly $C_{a+b+1,a+1}C_{c+d+1,c+1}$
by \cite{connolly2014}.
Let $i=b+1$ and $j=c+1$. Place $\sigma[1,b+1]$ (that is, $\sigma$ restricted to the first $b+1$ elements) inside the rectangle $[1,i] \times [\pi(i),n]$, and place $\tau[c+1,c+d+1]$ inside the rectangle $[j,n] \times [1,\pi(j)]$. This ensures that the points $(i,\pi(i))$ and $(j,\pi(j))$ are in their correct positions.

Now, notice that both $\sigma[b+2,a+b+1]$ and $\tau[1,c]$ form decreasing sequences, and have lengths $a$ and $c$ respectively. These two parts may be horizontally interleaved arbitrarily in between $i$ and $j$ in $\binom{a+c}{a}$ ways. Thus, we construct a total of $\binom{a+c}{c} C_{a+b+1,a+1}C_{c+d+1,c+1}$ this way. Ranging $a,b,c,d$ over compositions of $n-2$, we obtain all pairs $(\pi,(ij))$ of Type 4 exactly once. This completes the proof.
\end{proof}

This completes the proof of \cref{lem:p-intersections}, which implies \cref{lem:pairs-closed-form} and thus Theorem \ref{theorem:rss}.

\section{The Greedy Algorithm}
\label{sec:greedy}

In this section we define a greedy algorithm for finding twins in binary words, and use it to prove \cref{theorem:bss} and \cref{theorem:greedytwins}.

Given $s \in \{0,1\}^{2n}$, the greedy algorithm outputs a tableau (see \cref{def:tableaux}) $\tg(s) = (I,J)$ of semilength $n$. Here, $I = (i_1,\ldots, i_m)$ and $J = (j_1,\ldots, j_o)$ with $i_1 < \cdots < i_m$ and $j_1 < \dots < j_o$ and $i_k \neq j_r$ for all $k$ and $r$. Writing $i_r$ for the $r$-th term in $I$ and $j_r$ for the $r$-th term in $J$, this tableau has the property that for any $1\le r \le |J|$, $i_r < j_r$ and $s_{i_r} = s_{j_r}$. In particular, if $\tg(s)$ is rectangular, then this tableau exhibits $s$ as a shuffle square. In any case, the subsequence of $s$ indexed by the first $|J|$ elements of $I$ is identical and disjoint from the subsequence indexed by $J$, and so $\delta(s) \le \dg(s)\coloneqq |I| - |J|$. 

We now describe the algorithm explicitly. Initialize $r = 0$ and $U_0 = [2n]$, the set of unused bits. On iteration $1\le r \le n$, let $i_r \coloneqq \min (U_{r-1})$ and 
\begin{equation}\label{eq:greedy-iter}
j_r = \min \{j > j_{r-1}: j \ne i_r \textnormal{ and } s_j = s_{i_r}\},
\end{equation}
where $j_0 \coloneqq 0$. This continues until there is no valid choice for $j_r$, in which case the remaining unused indices in $[2n]$ are placed in sorted order at the end of $I$. In words, each step of the algorithm picks the next unused bit in $s$ for $i_r$ and the first matching copy of this bit to the right of $j_{r-1}$ for $j_r$. The following pseudocode details the decision-making process and return values of this algorithm.
\begin{algorithm}
\renewcommand{\thealgorithm}{}
\caption{The greedy algorithm}\label{greedy pseudocode}
\begin{algorithmic}
\State $r \gets 0$
\State $U_0 \gets [2n]$ \Comment{set of unused bits}
\While{$1 \leq r \leq n$}
\State $i_r \gets \min(U_{r-1})$
\State $j_r \gets \min\{j > j_{r-1}: j\neq i_r \text{ and } s_j = s_{i_r}\}$
\If{$j_r$ \text{ does not exist}}
    \State $I \gets (i_1,\ldots, i_r) \cup \textnormal{sorted}(U_{r-1})$
    \State $J \gets (j_1,\ldots, j_r)$
    \State \Return $(I, J)$ \Comment{algorithm terminates}
\Else
    \State $U_{r} = U_{r-1} \setminus \{i_{r+1}, j_{r+1}\}$
\EndIf
\EndWhile
\State $I \gets (i_1,\ldots, i_n)$
\State $J \gets (j_1,\ldots, j_n)$
\State \Return $(I, J)$
\end{algorithmic}
\end{algorithm}

\noindent If the algorithm succeeds to find $j_r$ for all $n$ iterations, it outputs a rectangular tableau exhibiting $s$ as a shuffle square.

\begin{example}
If $s = 1000110100$, we obtain $\tg(s) = ((1,2,3,4,6,8), (5,7,9,10))$, so that the underlined bits in $ \underline{1000}\overline{1}\underline{1}\overline{0}\underline{1}\overline{00}$ go into $I$ and the overlined bits go into $J$. The bits at positions $i_1$ through $i_4$ match the bits at positions $j_1$ through $j_4$, so this proves that $\delta(s)\le \dg(s) = 2$.
\end{example}

We are ready to prove our main lemma in the analysis of the greedy algorithm, which enumerates the number of words $s\in \{0,1\}^{2n}$ with any given value of $\dg(s)$. Note that since $|I| + |J| = 2n$ and $\dg(s) = |I| - |J|$, $\dg(s)$ is always even.

\begin{lemma}\label{lemma:greedy-deficit}
For $0\le i \le n$, the number of words $s\in \{0, 1\}^{2n}$ for which $\dg(s) =  2i$ is $\binom{2n}{n}$ if $i=0$ and $2\binom{2n}{n+i}$ otherwise.
\end{lemma}
\begin{proof}
Let $S_i$ denote the family of words $s\in \{0,1\}^{2n}$ for which $\dg(s)=2i$, and let $P_i$ denote the family of UD paths of semilength $n$ ending at $(2n, \pm 2i)$. Clearly,
\[
|P_i| = 
    \begin{cases} 
    \binom{2n}{n} & \text{ for } i = 0 \\
    2\binom{2n}{n+i} & \text{ for } i > 0,
    \end{cases}
\]
so it suffices to show that $|S_i| = |P_i|$ for all $0\le i \le n$.

Let $A_i$ be the family of UD paths of semilength $n$ never going below the $x$-axis that end at $(2n, 2i)$. Instead of proving a direct bijection between $S_i$ and $P_i$, we construct maps from both $S_i$ and $P_i$ onto $A_i$. We then show that the fibers of these two maps $\phi_S:S_i \rightarrow A_i$ and $\phi_P:P_i \rightarrow A_i$ over any given $p\in A_i$ have the same size $2^{x(p)}$, where $x(p)$ is the number of times $p$ intersects the $x$-axis (including the starting point $(0,0)$ but not the ending point if $p$ ends at $(2n,0)$), and this would complete the proof. We now construct and analyze $\phi_S$ and $\phi_P$ separately.
\vspace{3mm}

\noindent {\bf The map $\phi_S$.} If $s\in S_i$, this implies that the tableau $\tg(s) = (I,J)$ satisfies $|I| = |J| + 2i$. Since $I\sqcup J = [2n]$, we may define a UD path $\phi_S(s)$ from $(I,J)$ by taking an up-step on the indices in $I$, and a down-step on the indices in $J$. Since $(I, J)$ is a standard Young tableau, $\phi_S(s)$ never goes below the $x$ axis, and since $|I| = |J| + 2i$, the resulting path $\phi_S(s)$ ends at $(2n, 2i)$. This proves that $\phi_S$ is a well-defined map from $S_i$ to $A_i$.

The key difficulty is computing the sizes of the fibers of $\phi_S$. Given $p\in A_i$, there is a unique tableau $(I,J)$ where $I$ indexes the up-steps in $p$ and $J$ the down-steps. Our goal is to show that the number of $s \in \{0,1\}^{2n}$ for which $\tg(s) = (I,J)$ is exactly $2^{x(p)}$. We construct such $s$ by retracing the steps of the greedy algorithm.

At each iteration $1\le r\le |J|$, we pick the value of $s_{i_r} = s_{j_r}$ as follows. The value of $j_r$ must satisfy the iteration rule
\[
j_r = \min \{j > j_{r-1}: j \ne i_r \textnormal{ and } s_j = s_{i_r}\}. 
\]
We break into two cases based on whether or not $i_r > j_{r-1}$.

If $i_r = \min(U_{r-1}) > j_{r-1}$, this implies that all the bits before $j_{r-1}$ are used by iteration $r$, so $\{i_1,\ldots, i_{r-1}\} \sqcup \{j_1,\ldots, j_{r-1}\} = [2r-2]$. Call such an $r$ a ``pivot.'' In the original path $p \in A_i$, a pivot $r$ corresponds to a point in the path where there have been exactly $r-1$ up-steps and down-steps up to this point, so pivots are exactly the indices where $p$ leaves the $x$-axis. Whenever we reach a pivot, pick $s_{i_r} \in \{0,1\}$ arbitrarily, and let $s_{j_r} = s_{i_r}$. This accounts for the total number $2^{x(p)}$ of choices, $2$ independent choices for each pivot.

Otherwise, $r$ is not a pivot and $i_r < j_{r-1}$. In this case, let $k\le r-1$ be the unique index for which $j_{k-1} < i_r < j_k$, where we define $j_0 \coloneqq 0$. Since $i_k < i_r$, $j_k$ must be the smallest index after $i_r - 1$ for which $s_{j_k} = s_{i_k}$. In particular, $s_{i_r} \ne s_{j_k}$, so the bit $s_{i_r} = \overline{s_{j_k}}$ is uniquely determined. Note that $\overline{s_{j_k}}$ indicates that the bit $s_{j_k}$ has been flipped, i.e. 1 to 0 and 0 to 1. This operation can also be considered logically as the NOT operation or numerically as adding 1 to the original binary number.

We have now chosen the values of $s_{i_r}$ and $s_{j_r}$ for $r\le |J|$. It remains to choose the values of $s_{i_r}$ for $|J| < r \le |I|$. These $r$ are the remaining indices in $U_r$ after there no longer exists a valid choice for $j$. There are again two cases.

Suppose first that there is some $|J| < r \le |I|$ for which $i_r < j_{|J|}$. For any such $r$, find $k$ for which $j_{k-1} < i_r < j_k$ and define $s_{i_r} = \overline{s_{j_k}}$. For any such $r$, the values of $s_{i_{r'}}$ with $i_{r'} > j_{|J|}$ are uniquely determined to equal $\overline{s_{i_{r}}}$. 

Otherwise, if there is no $|J| < r \le |I|$ for which $i_r < j_{|J|}$, but $|I| > |J|$, then $\{i_1,\ldots, i_{|J|}\} \sqcup \{j_1,\ldots, j_{|J|}\} = [2|J|]$, so $|J|+1$ is a pivot. Pick $s_{i_{|J|+1}} \in \{0,1\}$ arbitrarily and the remaining $s_{i_r}$, $r > |J| + 1$ to be the opposite bit.

It is not difficult to check that at total of $2^{x(p)}$ binary words can arise from the above process, and each of them lies in $\phi_S^{-1}(p)$. This is because there are $x(p)$ pivots and $2$  choices for each pivot. After picking the pivots, the rest of the word can be uniquely constructed from the path $p \in A_i$. This completes the analysis of $\phi_S$.

\vspace{3mm}
\noindent {\bf The map $\phi_P$.} 
Given a UD path $p \in P_i$, divide it into segments above the $x$-axis and segments below the $x$-axis. Define $\phi_P(p)$ to be the UD path obtained by $p$ by reflecting each segment below the $x$-axis across the $x$-axis. The resulting path does not go below the $x$-axis, so it lies in $A_i$. Also, for any $p\in A_i$, the fiber $\phi_P^{-1}(p)$ has $2^{x(p)}$ elements corresponding to $2^{x(p)}$ ways to choose whether or not each of $x(p)$ segments of $p$ is reflected across the $x$-axis. This completes the proof that the fibers of $\phi_S$ and $\phi_P$ have the same size.
\end{proof}

We are now ready to prove \cref{theorem:bss} and \cref{theorem:greedytwins}.

\begin{proof}[Proof of \cref{theorem:bss}.]
By \cref{lemma:greedy-deficit}, the number of $s\in \{0,1\}^{2n}$ for which $\dg(s) = 0$ is $\binom{2n}{n}$. All such $s$ are shuffle squares, so there are at least $\binom{2n}{n}$ binary shuffle squares. For $n\ge 3$, note that the word $s_n = 1^{n-1} 0 1^{n-1} 0 \in \{0,1\}^{2n}$ is a binary shuffle square for which $\dg(s_n) > 0$, proving the theorem.
\end{proof}

\begin{proof}[Proof of \cref{theorem:greedytwins}.]

Since $\delta(s) \le \dg(s)$ for any word $s$, it suffices to show that almost all $s\in \{0,1\}^{2n}$ satisfy $\dg(s) \le h(n)\sqrt{n}$. By \cref{lemma:greedy-deficit}, we know that the number of $s$ for which $\dg(s) = 2i$ is $\binom{2n}{n}$ if $i=0$ and $2\binom{2n}{n+i}$ if $1\le i \le n$. Summing over $i\le a$ for some constant $C>0$, we obtain that for a uniform random $S\in \{0,1\}^{2n}$,
\[
\Pr[\dg(S) > 2a] = \Pr[|\bin(2n,1/2) - n| > a] \le 2 e^{-a^2/(24n)}
\]
by the Chernoff bound (see e.g. \cite[Theorem A.1.1]{alon2016}). Taking $a=\frac{1}{2}h(n)\sqrt{n}$, we see that 
\[
\Pr[\delta(S) > h(n)\sqrt{n}] \le \Pr[\dg(S) > h(n)\sqrt{n}] \le 2e^{-h(n)^2/96} \rightarrow 0,
\] 
as desired.
\end{proof}

\end{document}